\documentclass[preprint,12pt]{elsarticle}

\usepackage{amssymb}
\usepackage{url}
\usepackage{amsthm}
\usepackage{amsmath}

\newtheorem{thm}{\bf Theorem}
\newtheorem{prop}{\bf Property}
\newtheorem{propo}{\bf Proposition}
\newtheorem{lem}{\bf Lemma}

\journal{}

\begin{document}

\begin{frontmatter}

\title{The mixed metric dimension of flower snarks and wheels}

\author{Milica Milivojevi\'c Danas\fnref{kg}}
\ead{milica.milivojevic@kg.ac.rs}

\address[kg]{Faculty of Science and Mathematics, University of Kragujevac, Radoja Domanovi\' ca 12, Kragujevac, Serbia }

\begin{abstract}
New graph invariant, which is called mixed metric dimension, has been recently introduced.
In this paper, exact results of mixed metric dimension on two special classes  of graphs are found:
flower snarks $J_n$ and wheels $W_n$. It is proved that mixed metric dimension for $J_5$ is equal to 5,
while for higher dimensions it is constant and equal to 4. For $W_n$, its mixed metric dimension is not constant,
but it is equal to $n$ when $n\geq 4$, while it is equal to 4, for $n=3$.
\end{abstract}


\end{frontmatter}

 \section{Introduction}

Let $G=(V,E)$ be a connected graph, where $V$ represent a set whose elements are called vertices and $E$ represent a set whose elements are called edges. The mixed metric dimension of graphs was introduced by Kelenc et al. (2017) in \cite{bib8}. This dimension of a graph G is a combination the metric dimension and edge metric dimension.

The distance in the connected graph $G$, for any two vertices $u$ and $v$, is the length of a shortest $u-v$ path in $G$. The vertex $w$ resolving a pair $u,v \in V$ if $d(u, w)\neq d(v, w)$.  The metric coordinates $r(v,S)$ of vertex $v$ with respect to an ordered set of vertices $S=\{w_{1},w_{2},...,w_{k}\}$ is defined as $r(v,S)=(d(v,w_{1}),d(v,w_{2}),$ $...,d(v,w_{k}))$. Set $S$ is called resolving set if metric coordinates of all vertices differ  from each other in respect to set $S$. The metric basis of graph $G$ is a resolving set of the minimum cardinality. The metric dimension of graph $G$ is the cardinality of metric basis for graph  $G$ and is denoted by $\beta(G)$. Slater in \cite{bib1} and Harary and Melter (1976) in \cite{bib2} independently of one another introduced resolving sets for graphs. Also, there were published several works regarding applications and some theoretical properties of this invariante. For instance, applications to the direction of robots in networks are analized in \cite{bib3} and applications to chemistry in (\cite{bib5,bib16}), among others. In the literature exist several other variations of metric dimension:  resolving dominating sets \cite{bib13} , strong metric dimension \cite{bib14},  local metric dimension \cite{bib12}, $k$-metric dimension \cite{bib15}, $k$-metric antidimension  \cite{bib18,bib17}, etc.

The distance between vertex $w$ and edge $uv$ of graph $G$ is defined as $d(uv, w)=\min \{d(u,w), d(v,w)\}.$ The vertex $w$ edge resolving a pair $e,f \in E$ if $d(w, e)\neq d(w, f)$. The metric coordinates $r(e,S)$ of edge $e$ with respect to an ordered set of vertices $S=\{w_{1},w_{2},...,w_{k}\}$ is defined as $r(e,S)=(d(e,w_{1}),d(e,w_{2}),...,d(e,w_{k}))$.  Set $S$ is edge resolving set if metric coordinates of all edges differ  from each other in respect to set $S$. The edge metric basis of graph G is a edge resolving set
of the minimum cardinality. The edge metric dimension of $G$ is the cardinality of  edge metric basis  for graph  $G$ and is denoted by $\beta_{E}(G)$. The concept of edge metric dimension of graph $G$ was introduced by Kelenc at al., (2016) in \cite{bib11}.

For given graph $G$, since,  every vertex of graph is uniquely determined by resolving set of a graph   and   every edge of graph is uniquely determined by edge resolving set of a graph, the logical question is: whether every edge resolving set of a graph $G$ is also a resolving set and vice versa? In paper \cite{bib11}, authors proved there are several graph families for which the edge resolving set is also a resolving set the graph, but in general case, it is not valid for every graph $G.$ Similarly, for every graph $G$ resolving set is not necessarily edge resolving set for $G$.

Let define set of items as  $V\cup E$, i.e. each item is vertex or edge.
The vertex $v$ mixed resolving a pair of items if $d(v, a)\neq d(v, b)$. The metric coordinates $r(a,S)$ of item $a$ with respect to an ordered set of vertices $S=\{w_{1},w_{2},...,w_{k}\}$ is defined as $r(a,S)=(d(a,w_{1}),d(a,w_{2}),...,d(a,w_{k}))$.  Set $S$ is mixed  resolving set if metric coordinates of all items differ from each other in respect to set $S$. The mixed metric basis of graph G is a mixed resolving set of the minimum cardinality. The mixed metric dimension of $G$ is the cardinality of mixed metric basis for graph  $G$ and is denoted by $\beta_{M}(G)$.

In literature mixed metric dimension is known for several well-known classes of graphs
\begin{itemize}
\item Kelenc et al. (2017) \cite{bib8}: path, cycle, tree, grid, complete bipartite graph;
\item Raza et al. (2019) \cite{raz19}: prism, anti-prism and graph of convex polytopes $R_n$.
\end{itemize}

In this paper, this dimension  will be studied  for two special classes of graphs: flower snarks and wheels.

A flower snark is connected,  bridgeless 3-regular graph. These graphs are denoted with  $J_n$ and have $4n$ vertices and $6n$ edges, where vertex-set is $V(J_n)=\{a_i,b_i,c_i,d_i|i=0,...,n-1\}.$ Vertices $\{a_i|0\leq i \leq n-1\}$ are called as inner vertices and they induce the inner cycle, while set of vertices  $\{b_i|0\leq i \leq n-1\}$ are called a set of central vertices and $\{c_i|0\leq i \leq n-1\}$ and $\{d_i|0\leq i \leq n-1\}$ are outer vertices and they induce the outer cycle. Let edge-set is $E(J_n)=\{a_ib_i,a_ic_i,a_id_i,b_ib_{i+1},c_ic_{i+1},d_id_{i+1}| i=0,1,...,n-1\},$ where the edges $a_{n-1}a_0$ and ${b_{n-1}b_0}$ are replaced by edges $a_{n-1}b_0$ and ${b_{n-1}a_0}$. The indices are taken modulo $n$.

\begin{prop} \label{iso} \mbox{\rm(\cite{kra20})} Let $J_n$ be a flower snark graph and $j\in \{0,1,...,n-1\}$ be an arbitrary number. Then the function $h_j:V(J_n)\rightarrow V(J_n)$ defined as:

\begin{equation}
\begin{aligned}
h_j(a_i) = \begin{cases}
a_{j-i}, & i \le j \le n-1\\
a_{n+j-i}, & 0 \le j < i
\end{cases} \\
h_j(b_i) = \begin{cases}
b_{j-i}, & i \le j \le n-1\\
b_{n+j-i}, & 0 \le j < i
\end{cases}\\
h_j(c_i) = \begin{cases}
c_{j-i}, & i \le j \le n-1\\
d_{n+j-i}, & 0 \le j < i
\end{cases}\\
h_j(d_i) = \begin{cases}
d_{j-i}, & i \le j \le n-1\\
c_{n+j-i}, & 0 \le j < i
\end{cases}
\end{aligned}
\end{equation}
is an isomorphism of flower snark graph $J_n$.
\end{prop}

A wheel graph is a cycle of length at least 3, with a single vertex in the center connected to every vertex on the cycle. These graphs are denoted with $W_{n}$ and  have $n+1$ vertices and $2n$ edges, where vertex-set is $ V(W_n)=\{v_{0},v_{1},...,v_{n}\}$ and edge-set is $E(W_n)=\{v_0v_i| 1\leq i \leq n\}\cup \{v_iv_{i+1}|1\leq i \leq n-1\}\cup \{v_n v_1\}$. Vertex $v_0$ is called as interior vertex of the graph, and all other vertices are called external vertices.

In Figure 1 is presented flower snark   $J_9$. Its mixed metric dimension is 4, which is obtained through total enumeration. The one mixed metric basis is   $\{b_0, c_1, c_6, d_3\}.$

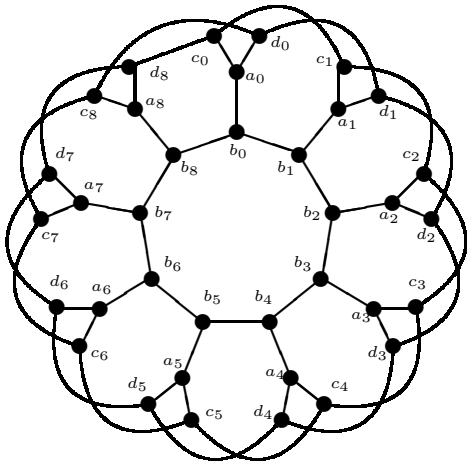
\begin{figure}[htbp]
\centering\setlength\unitlength{1mm}
\begin{picture}(56,56)
\thicklines
\tiny
\put(28.0,48.3){\circle*{2}} \put(28.0,40.3){\circle*{2}}
\put(25.0,53.1){\circle*{2}}
\put(31.0,53.1){\circle*{2}}
\put(29.1,47.2){$a_{0}$} \put(27.0,37.3){$b_{0}$}
\put(21.9,49.7){$c_{0}$} \put(32.5,51.7){$d_{0}$}
\put(28.0,48.3){\line(0,-1){8.0}} \put(28.0,40.3){\line(3,-1){8.4}} \put(28.0,48.3){\line(-3,5){3.0}} \put(28.0,48.3){\line(3,5){3.0}} \qbezier(24.954510,53.123381)(36.410992,62.608941)(42.260146,49.021874) \qbezier(30.991336,53.123381)(44.406972,59.698645)(46.884626,45.141479) \put(41.5,43.4){\circle*{2}} \put(36.3,37.3){\circle*{2}}
\put(42.3,49.0){\circle*{2}}
\put(46.9,45.1){\circle*{2}}
\put(41.4,41.2){$a_{1}$} \put(33.4,35.0){$b_{1}$}
\put(38.5,49.5){$c_{1}$} \put(46.9,42.5){$d_{1}$}
\put(41.5,43.4){\line(-5,-6){5.1}} \put(36.3,37.3){\line(3,-5){4.4}} \put(41.5,43.4){\line(0,1){5.6}} \put(41.5,43.4){\line(3,1){5.4}} \qbezier(42.260146,49.021874)(57.355591,48.833476)(52.880642,34.756095) \qbezier(46.884626,45.141479)(61.610169,41.464346)(53.928933,28.810982) \put(48.7,30.9){\circle*{2}} \put(40.8,29.6){\circle*{2}}
\put(52.9,34.8){\circle*{2}}
\put(53.9,28.8){\circle*{2}}
\put(46.9,28.7){$a_{2}$} \put(36.8,29.0){$b_{2}$}
\put(50.0,36.9){$c_{2}$} \put(51.9,26.1){$d_{2}$}
\put(48.7,30.9){\line(-6,-1){7.9}} \put(40.8,29.6){\line(-1,-6){1.5}} \put(48.7,30.9){\line(1,1){4.2}} \put(48.7,30.9){\line(5,-2){5.3}} \qbezier(52.880642,34.756095)(64.545401,24.817933)(51.846559,17.001150) \qbezier(53.928933,28.810982)(63.067816,16.438061)(48.828155,11.773099) \put(46.2,16.8){\circle*{2}} \put(39.2,20.8){\circle*{2}}
\put(51.8,17.0){\circle*{2}}
\put(48.8,11.8){\circle*{2}}
\put(43.2,15.5){$a_{3}$} \put(35.6,22.3){$b_{3}$}
\put(50.8,19.9){$c_{3}$} \put(45.4,10.4){$d_{3}$}
\put(46.2,16.8){\line(-5,3){6.9}} \put(39.2,20.8){\line(-5,-4){6.8}} \put(46.2,16.8){\line(1,0){5.7}} \put(46.2,16.8){\line(1,-2){2.7}} \qbezier(51.846559,17.001150)(54.616235,1.799433)(39.641754,4.064761) \qbezier(48.828155,11.773099)(48.097865,-3.670152)(33.968998,2.000032) \put(35.2,7.6){\circle*{2}} \put(32.4,15.1){\circle*{2}}
\put(39.6,4.1){\circle*{2}}
\put(34.0,2.0){\circle*{2}}
\put(31.8,7.9){$a_{4}$} \put(30.4,17.9){$b_{4}$}
\put(40.5,6.3){$c_{4}$} \put(30.2,2.5){$d_{4}$}
\put(35.2,7.6){\line(-2,5){2.7}} \put(32.4,15.1){\line(-1,0){8.9}} \put(35.2,7.6){\line(5,-4){4.5}} \put(35.2,7.6){\line(-1,-5){1.2}} \qbezier(39.641754,4.064761)(32.214053,-9.451429)(21.976983,2.000000) \qbezier(33.968998,2.000032)(23.704910,-9.451451)(16.304215,4.064700) \put(20.8,7.6){\circle*{2}} \put(23.5,15.1){\circle*{2}}
\put(22.0,2.0){\circle*{2}}
\put(16.3,4.1){\circle*{2}}
\put(18.2,9.3){$a_{5}$} \put(23.6,17.9){$b_{5}$}
\put(23.8,2.5){$c_{5}$} \put(13.5,6.3){$d_{5}$}
\put(20.8,7.6){\line(2,5){2.7}} \put(23.5,15.1){\line(-5,4){6.8}} \put(20.8,7.6){\line(1,-5){1.2}} \put(20.8,7.6){\line(-5,-4){4.5}} \qbezier(21.976983,2.000000)(7.821067,-3.670259)(7.117774,11.772989) \qbezier(16.304215,4.064700)(1.302668,1.799292)(4.099342,17.001023) \put(9.8,16.8){\circle*{2}} \put(16.7,20.8){\circle*{2}}
\put(7.1,11.8){\circle*{2}}
\put(4.1,17.0){\circle*{2}}
\put(8.7,19.2){$a_{6}$} \put(18.3,22.3){$b_{6}$}
\put(8.6,10.4){$c_{6}$} \put(3.1,19.9){$d_{6}$}
\put(9.8,16.8){\line(5,3){6.9}} \put(16.7,20.8){\line(-1,6){1.5}} \put(9.8,16.8){\line(-1,-2){2.7}} \put(9.8,16.8){\line(-1,0){5.7}} \qbezier(7.117774,11.772989)(-7.148991,16.437875)(2.016905,28.810845) \qbezier(4.099342,17.001023)(-8.626620,24.817739)(3.065164,34.755963) \put(7.3,30.9){\circle*{2}} \put(15.2,29.6){\circle*{2}}
\put(2.0,28.8){\circle*{2}}
\put(3.1,34.8){\circle*{2}}
\put(7.7,32.8){$a_{7}$} \put(17.1,29.0){$b_{7}$}
\put(2.0,26.1){$c_{7}$} \put(3.9,36.9){$d_{7}$}
\put(7.3,30.9){\line(6,-1){7.9}} \put(15.2,29.6){\line(3,5){4.4}} \put(7.3,30.9){\line(-5,-2){5.3}} \put(7.3,30.9){\line(-1,1){4.2}} \qbezier(2.016905,28.810845)(-5.691477,41.464168)(9.061126,45.141379) \qbezier(3.065164,34.755963)(-1.436938,48.833320)(13.685585,49.021798) \put(14.5,43.4){\circle*{2}} \put(19.6,37.3){\circle*{2}}
\put(9.1,45.1){\circle*{2}}
\put(13.7,49.0){\circle*{2}}
\put(15.8,43.9){$a_{8}$} \put(20.5,35.0){$b_{8}$}
\put(7.1,42.5){$c_{8}$} \put(16.4,47.7){$d_{8}$}
\put(14.5,43.4){\line(5,-6){5.1}} \put(19.6,37.3){\line(3,1){8.4}} \put(14.5,43.4){\line(-3,1){5.4}} \put(14.5,43.4){\line(0,1){5.6}} \qbezier(9.061126,45.141379)(16.039347,59.698570)(30.991336,53.123381) \qbezier(13.685585,49.021798)(19.320048,51.072590)(24.954510,53.123381) \end{picture}
\caption{Graph $J_{9}$}
\end{figure}

In Figure 2 is presented wheel graph $W_8$. Its mixed metric dimension is 8, which is obtained through total enumeration. The one mixed metric basis is $\{ v_1, v_2, v_3, v_4, v_5,v_6,v_7,v_8\}.$ In the below figure,  vertices that are elements of the basis are shown in larger circles.

\begin{figure}[htbp]
\centering\setlength\unitlength{1mm}
\begin{picture}(68,68)
\thicklines
\tiny
\put(35.0,34.0){\circle*{2}}  \put(36.4,65.8){$v_{1}$} \put(34.9,63.9){\line(5,-2){21.2}} \put(34.9,63.9){\line(0,-1){29.9}}  \put(58.2,54.7){$v_{2}$} \put(56.1,55.1){\line(2,-5){8.8}} \put(56.1,55.1){\line(-1,-1){21.1}} \put(65.9,35.2){$v_{3}$} \put(64.9,33.9){\line(-2,-5){8.8}} \put(64.9,33.9){\line(-1,0){29.9}} \put(51.7,9.6){$v_{4}$} \put(56.1,12.7){\line(-5,-2){21.2}} \put(56.1,12.7){\line(-1,1){21.1}}  \put(31.4,2.0){$v_{5}$} \put(34.9,3.9){\line(-5,2){21.2}} \put(34.9,3.9){\line(0,1){30.1}}  \put(9.3,13.1){$v_{6}$} \put(13.7,12.7){\line(-2,5){8.8}} \put(13.7,12.7){\line(1,1){21.3}} \put(0,35.2){$v_{7}$} \put(4.9,33.9){\line(2,5){8.8}} \put(4.9,33.9){\line(1,0){30.1}} \put(13.1,58.2){$v_{8}$} \put(13.7,55.1){\line(5,2){21.2}} \put(13.7,55.1){\line(1,-1){21.3}}
 \put(32.2,36.9){$v_{0}$}
\put(35,64.0){\circle*{2}}
\put(35,64.0){\circle{3}}
 \put(56.1,55.1){\circle*{2}}
  \put(56.1,55.1){\circle{3}}
   \put(64.5,33.9){\circle*{2}}
     \put(64.5,33.9){\circle{3}}

     \put(56.1,12.7){\circle*{2}}
      \put(56.1,12.7){\circle{3}}
      \put(35,3.9){\circle*{2}}
       \put(35,3.9){\circle{3}}
      \put(13.7,12.7){\circle*{2}}
         \put(13.7,12.7){\circle{3}}
          \put(4.9,33.9){\circle*{2}}
           \put(4.9,33.9){\circle{3}}
            \put(13.7,55.1){\circle*{2}}
                \put(13.7,55.1){\circle{3}}
\end{picture}
\caption{Graph $W_{8}$}
\end{figure}
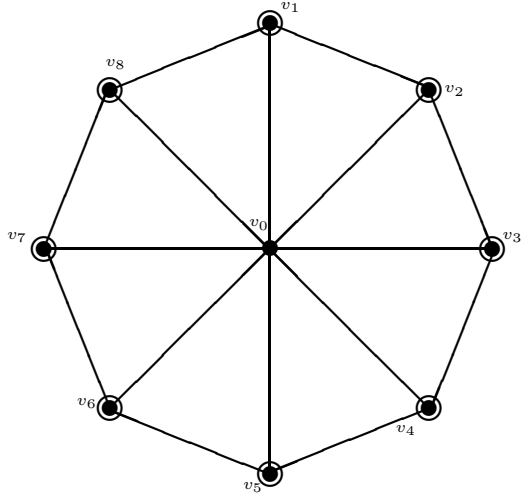

Next it be presented some theoretical properties of metric dimension, edge metric dimension and mixed metric dimension known for in literature,
which is used in the next section.

\begin{prop} \mbox{\rm(\cite{bib8})}
For \label{amc} any graph $G$ it holds $$\beta_{M}(G)\geq \max\{\beta(G),\beta_{E}(G)\}.$$
\end{prop}
 In the next theorem \cite{bib20}, for  the flower snarks, metric dimension $J_n$ is given.
\begin{thm}\label{metd} \mbox{\rm(\cite{bib20})}
Let $J_n$ be the flower snark. Then for every odd positive integer $n\geq 5$ it holds $\beta(J_n)= 3.$
\end{thm}

\begin{propo}\label{exc} \mbox{\rm(\cite{bib8})}
Let $v$ be an arbitrary vertex in a graph $G$ and let $S=V(G)\setminus \{v\}.$ If $(\forall w\in N(v))\, (\exists x\in S) \,\, d(vw,x)\neq d(w,x)$, then $S$ is a mixed resolving set for the graph $G$.
\end{propo}

\begin{propo} \mbox{\rm(\cite{bib7})} \label{edim1} Let $G$ be a connected graph, then $\beta_E(G) \geq 1 + \lceil log_2 \delta(G) \rceil$.\end{propo}

\section{\bf  The new theoretical results}

In this section, it will be consider the mixed metric dimension of two important classes of graphs: flower snarks and wheels graph.

\subsection{The mixed metric dimension of flower snarks}

The metric dimension  for flower snarks is given in paper \cite{bib20}. So it was interesting to determine the value for mixed metric dimension for these graphs.
In the Theorem \ref{fs}, the result about the mixed metric dimension is given and here it is shown that the dimension is different for $n=5$ and odd $n\geq 7$.
It should be noted we will omitted all cases which are equivalent with regards to isomorphism given in Property \ref{iso}.

First, let set of vertices $V(J_n)$ be partitioned as follows:
$V_{1}=\{a_{i}| 0\leq i \leq n-1\}$, $V_2=\{b_{i}| 0\leq i \leq n-1\}$, $V_3=\{c_{i}| 0\leq i \leq n-1\}$ and $V_4=\{d_{i}| 0\leq i \leq n-1\}.$

Next, it will be proposed and proved several lemas which help us to prove Theorem \ref{fs}.

\begin{lem}\label{lem1} If $S$ is an arbitrary mixed resolving set of $J_n$, then: \\
a) $S \bigcap (V_1 \bigcup V_2) \neq \emptyset$; \\
b) $S \bigcap (V_3 \bigcup V_4) \neq \emptyset$; \\
c) $(\forall i \in \{0,1,...,n-1\})$  $S \bigcap \{a_j,b_j,c_j,d_j | i \leq j \leq i+k-1\} \neq \emptyset.$ \\

\end{lem}
\begin{proof} a)  Suppose the opposite, i.e.  $S \bigcap (V_1 \bigcup V_2) = \emptyset$,
which means that $S \subseteq (V_3 \bigcup V_4)$. For each $i$ and $j$, such that $0 \leq i,j \leq n-1$, it holds
$d(b_j,c_i) =d(b_j,d_i)=d(a_j,c_i)+1$ so $d(a_jb_j,c_i) =d(a_j,c_i)=d(a_jb_j,d_i) =d(a_j,d_i)$
which means that edge $a_jb_j$ has the same metric coordinates as vertex $a_j$
with respect to $V_3 \bigcup V_4$, i.e. $r(a_jb_j,V_3 \bigcup V_4) = r(a_j,V_3 \bigcup V_4)$.
Since $S \subseteq (V_3 \bigcup V_4)$, then holds $r(a_jb_j,S) = r(a_j,S)$,
implying that $S$ is not a mixed resolving set of $J_n$ which is a contradiction with a starting assumption. \\
b) Suppose the opposite, i.e.  $S \bigcap (V_3 \bigcup V_4) = \emptyset$,
which means that $S \subseteq (V_1 \bigcup V_2)$.
For each $i$ and $j$, such that $0 \leq i,j \leq n-1$ and $i \ne j$, it holds
$d(c_j,a_i) =d(d_j,a_i)=d(a_j,a_i)-1$ and $d(c_j,b_i) =d(d_j,b_i)=d(a_j,b_i)+1$.
For $i=j$ similarly $d(c_i,a_i) =d(d_i,a_i)=1$ and $d(c_i,b_i) =d(d_i,b_i)=2=d(a_i,b_i)+1$ holds.
Next, for each $j$, such that $0 \leq j \leq n-1$ it holds
that vertex $c_j$ has the same metric coordinates as vertex $d_j$
with respect to $V_1 \bigcup V_2$, i.e. $r(c_j,V_1 \bigcup V_2) = r(d_j,V_1 \bigcup V_2)$.
Since $S \subseteq (V_1 \bigcup V_2)$, then holds $r(c_j,S) = r(d_j,S)$,
implying that $S$ is not a mixed resolving set of $J_n$ which is a contradiction with a starting assumption. \\
c) Suppose the opposite, i.e.  $S \bigcap \{a_j,b_j,c_j,d_j | i \leq j \leq i+k-1\} = \emptyset$,
which means that $S \subseteq \{a_j,b_j,c_j,d_j | i-k-1 \leq j \leq i-1\}$.
For $i-k \leq j \leq i-1$ it holds $d(b_{i+1},a_j) = d(b_i,a_j)+1$,
$d(b_{i+1},b_j) = d(b_i,b_j)+1$, $d(b_{i+1},c_j) = d(b_i,c_j)+1$
and $d(b_{i+1},d_j) = d(b_i,d_j)+1$ so $d(b_ib_{i+1},a_j) = d(b_i,a_j)$,
$d(b_ib_{i+1},b_j) = d(b_i,b_j)$, $d(b_ib_{i+1},c_j) = d(b_i,c_j)$
and $d(b_ib_{i+1},d_j) = d(b_i,d_j)$ which means that
edge $b_ib_{i+1}$ has the same metric coordinates as vertex $b_i$
with respect to $\{a_j,b_j,c_j,d_j | i-k \leq j \leq i-1\}$.
For only remained case when $j=i-k-1$ it holds
$d(b_{i+1},a_{i-k-1}) = d(b_i,a_{i-k-1})-1$,
$d(b_{i+1},b_{i-k-1}) = d(b_i,b_{i-k-1})-1$, $d(b_{i+1},c_{i-k-1}) = d(b_i,c_{i-k-1})-1$
and $d(b_{i+1},d_{i-k-1}) = d(b_i,d_{i-k-1})-1$ so $d(b_ib_{i+1},a_{i-k-1}) = d(b_{i+1},a_{i-k-1})$,
$d(b_ib_{i+1},b_j) = d(b_{i+1},b_{i-k-1})$, \\
$d(b_ib_{i+1},c_{i-k-1}) = d(b_{i+1},c_{i-k-1})$
and $d(b_ib_{i+1},d_{i-k-1}) = d(b_{i+1},d_{i-k-1})$ which means that
edge $b_ib_{i+1}$ has the same metric coordinates as vertex $b_{i+1}$
with respect to $\{a_{i-k-1},b_{i-k-1},c_{i-k-1},d_{i-k-1}\}$.
Therefore, $r(b_ib_{i+1}, \{a_j,b_j,c_j,d_j | i-k-1 \leq j \leq i-1\}) = r(b_{i+1}, \{a_j,b_j,c_j,d_j | i-k-1 \leq j \leq i-1\})$.
Having in mind that $S \subseteq \{a_j,b_j,c_j,d_j | i-k-1 \leq j \leq i-1\}$,
we have $r(b_ib_{i+1}, S) = r(b_{i+1}, S)$, so $S$ is not a mixed resolving set of $J_n$ which is a contradiction with a starting assumption. \\
\end{proof}

It should be noted that, as is mentioned earlier, all indices in part c) are taken modulo n.
Moreover, without loss of generality, one vertex from $S \bigcap (V_3 \bigcup V_4)$,
from Lemma \ref{lem1}  part b) should be transformed into the vertex $c_0$ by isomorphism
from Property \ref{iso}.

\begin{thm} \label{fs} For odd $n \geq 5$, it holds $\beta_M(J_n) = \begin{cases}
 5, & n = 5  \\
 4, & n \geq 7 \\
\end{cases} $.\end{thm}
\begin{proof}
\underline{{\bf Step 1.}}  {\em Exact value for $n  \in \{5,7,9\}$} \\
By using total enumeration technique, it can be shown that:
\begin{itemize}
\item $\beta_M(J_5)  = 5$, with mixed metric basis $\{a_3, b_0, b_1, c_2, d_3\}$;
\item $\beta_M(J_7)  = 4$, with mixed metric basis $S = \{b_0, c_1, c_5, d_3\}$;
\item $\beta_M(J_9)  = 4$, with mixed metric basis $S = \{b_0, c_1, c_6, d_3\}$.
\end{itemize}

 \textbf{\underline{Step 2}:}  {\em Upper bound equals 4 for $n \geq 11$}.  \\
Let $S = \{b_0, c_1, c_{k+2}, d_3\}$. It will be proved that $S$ is mixed metric resolving set. The representation of coordinates of each vertex and each edge,with respect to $S$, will be shown in the Table \ref{jvert} and Table \ref{jedge}.
\newpage
\begin{table}
\begin{center}
\caption{  Metric coordinates of vertices of $J_{2k+1}$}
\label{jvert}
\begin{tabular}{|c|c|c|}
 \hline
  vetex & cond. & $r(v,S)$\\
 \hline
$a_{i}$ &  $0 \leq i \leq 1$ & ($i+1, 2-i, k+i, 4-i$) \\
$a_2$ &   & ($3, 2, k+1, 2$) \\
$a_{i}$ &  $3 \leq i \leq k$ & ($i+1, i, k+3-i, i-2$) \\
$a_{k+1}$ &   & ($k+1, k+1, 2, k-1$) \\
$a_{i}$  &  $k+2 \leq i \leq k+3$ & ($2k+2-i, 2k+3-i, i-1-k, i-2$) \\
   &  $k+4 \leq i \leq 2k$ & ($2k+2-i, 2k+3-i, i-1-k, 2k+5-i$) \\
 \hline
$b_{i}$ &  $0 \leq i \leq 1$ & ($i, 3-i, k+i+1, 5-i$) \\
$b_2$ &   & ($2, 3, k+2, 3$) \\
$b_{i}$ &  $3 \leq i \leq k$ & ($i, i+1, k+4-i, i-1$) \\
$b_{k+1}$ &   & ($k, k+2, 3, k$) \\
$b_{i}$  &  $k+2 \leq i \leq k+3$ & ($2k+1-i, 2k+4-i, i-k, i-1$) \\
   &  $k+4 \leq i \leq 2k$ & ($2k+1-i, 2k+4-i, i-k, 2k+6-i$) \\
 \hline
$c_0$ &   & ($2, 1, k+1, 5$) \\
$c_{i}$ &  $1 \leq i \leq 2$ & ($i+2, i-1, k+2-i, 5-i$) \\	
    &  $3 \leq i \leq k$ & ($i+2, i-1, k+2-i, i-1$) \\
		&  $k+1 \leq i \leq k+2$ & ($2k+3-i, i-1, k+2-i, i-1$) \\
		&  $k+3 \leq i \leq 2k$ & ($2k+3-i, 2k+4-i, i-2-k, 2k+4-i$) \\
 \hline
$d_0$ &   & ($2, 3, k-1, 3$) \\
$d_{i}$ &  $1 \leq i \leq 2$ & ($i+2, i+1, k-1+i, 3-i$) \\	
    &  $3 \leq i \leq k$ & ($i+2, i+1, k+4-i, i-3$) \\
		&  $k+1 \leq i \leq k+2$ & ($2k+3-i, 2k+2-i, k+4-i, i-3$) \\
		&  $k+3 \leq i \leq k+4$ & ($2k+3-i, 2k+2-i, i-k, i-3$) \\
		&  $k+5 \leq i \leq 2k$ & ($2k+3-i, 2k+2-i, i-k, 2k+6-i$) \\
	  \hline
\end{tabular}
\end{center}
\end{table}
\newpage

\begin{table}
\begin{center}
\caption{Metric coordinates of edges of $J_{2k+1}$}
\label{jedge}
\begin{tabular}{|c|c|c|}
 \hline
edge & cond. & $r(e,S)$\\
  \hline
$ a_0b_0$ &   & ($0, 2, k, 4$) \\
$a_ib_i$    & $1 \leq i \leq 2$  & ($i, i, k+1, 4-i$) \\
    & $3 \leq i \leq k$  & ($i, i, k+3-i, i-2$) \\
$ a_{k+1}b_{k+1}$ &   & ($k, k+1, 2, k-1$) \\
$a_ib_i$    & $k+2 \leq i \leq k+3$  & ($2k+1-i, 2k+3-i, i-1-k, i-2$) \\
    & $k+4 \leq i \leq 2k$  & ($2k+1-i, 2k+3-i, i-1-k, 2k+5-i$) \\
  \hline
$ a_0c_0$ &   & ($1, 1, k, 4$) \\
$a_ic_i$    & $1 \leq i \leq 2$  & ($i+1, i-1, k+2-i, 4-i$) \\
    & $3 \leq i \leq k$  & ($i+1, i-1, k+2-i, i-2$) \\
    & $k+1 \leq i \leq k+2$  & ($2k+2-i, i-1, k+2-i, i-2$) \\
    & $k+3 \leq i \leq 2k$  & ($2k+2-i, 2k+3-i, i-2-k, 2k+4-i$) \\
   \hline
$ a_0d_0$ &   & ($1, 2, k-1, 3$) \\
$a_id_i$    & $1 \leq i \leq 2$  & ($i+1, i, k-1+i, 3-i$) \\
    & $3 \leq i \leq k$  & ($i+1, i, k+3-i, i-3$) \\
$ a_{k+1}d_{k+1}$ &   & ($k+1, k+1, 2, k-2$) \\
$a_id_i$  & $k+2 \leq i \leq k+3$  & ($2k+2-i, 2k+2-i, i-1-k, i-3$) \\
    & $k+4 \leq i \leq 2k$  & ($2k+2-i, 2k+2-i, i-1-k, 2k+5-i$) \\
   \hline
$ b_0b_1$ &   & ($0, 2, k+1, 4$) \\
$b_{i}b_{i+1}$& $1 \leq i \leq 2$  & ($i, i+1, k+3-i, 4-i$) \\
    & $3 \leq i \leq k$  & ($i, i+1, k+3-i, i-1$) \\
    &  $k+1 \leq i \leq k+2$ & ($2k-i, 2k+3-i, 2, i-1$) \\
    & $k+3 \leq i \leq 2k$  & ($2k-i, 2k+3-i, i-k, 2k+5-i$) \\
  \hline
$ c_0c_1$ &   & ($2, 0, k+1, 4$) \\
$c_{i}c_{i+1}$& $1 \leq i \leq 2$  & ($i+2, i-1, k+1-i, 4-i$) \\
    & $3 \leq i \leq k$  & ($i+2, i-1, k+1-i, i-1$) \\
$ c_{k+1}c_{k+2}$ &   & ($k+1, k, 0, k$) \\
$c_{i}c_{i+1}$ & $k+2 \leq i \leq 2k-1$  & ($2k+2-i, 2k+3-i, i-2-k, 2k+3-i$) \\
  \hline
$ c_{2k}d_0$ &   & ($2, 3, k-2, 3$) \\
$ c_0d_{2k}$ &   & ($2, 1, k, 5$) \\
  \hline
$ d_0d_1$ &   & ($2, 2, k-1, 2$) \\
$d_{i}d_{i+1}$& $1 \leq i \leq 2$  & ($i+2, i+1, k-1+i, 2-i$) \\
    & $3 \leq i \leq k$  & ($i+2, i+1, k+3-i, i-3$) \\
$ d_{k+1}d_{k+2}$ &   & ($k+1, k, 2, k-2$) \\
$d_{i}d_{i+1}$ & $k+2 \leq i \leq k+3$  & ($2k+2-i, 2k+1-i, i-k, i-3$) \\
 & $k+4 \leq i \leq 2k-1$  & ($2k+2-i, 2k+1-i, i-k, 2k+5-i$) \\
   \hline
\end{tabular}
\end{center}
\end{table}
As it can be seen from Table \ref{jvert} and Table \ref{jedge},
all items have mutually different metric coordinates, so $S$ is a mixed metric resolving set.
Therefore, $\beta_M(J_n)  \leq 4$. \\

%%\newpage

\textbf{\underline{Step 3}:}  {\em Lower bound equals 4 for $n \geq 11$}.  \\
Suppose the opposite, i.e.  $\beta_{M}(J_{n}) \leq 3$.    By Property \ref{amc} and Theorem \ref{metd}  we have $\beta_{M}(J_{n})=3$.
Let $S$ be a mixed resolving set of $J_n$.
Then, from Lemma \ref{lem1}, part $a)$ and part $b)$ could be derived two members of set $S$, i.e. $c_0 \in S$ and
$(\exists i) (a_i \in S \vee b_i \in S)$. Let remaining third member of set $S$ has index $j$,
i.e. $(\exists j) (a_j \in S \vee b_j \in S \vee c_j \in S \vee d_j \in S)$.

From Lemma  \ref{lem1}, part $c)$, it holds that indices $i$ and $j$ have not arbitrary values:
\begin{itemize}
  \item [I)] $i\neq 0 \,\wedge\, j\neq 0;$
  \item [II)] $1\leq i\leq k \Rightarrow k+1\leq j\leq 2k;$
  \item [III)]$ k+1\leq i\leq 2k \Rightarrow 1\leq j\leq k. $

\end{itemize}

Part I) holds, because if $(i= 0 \, \vee \, j= 0) \Rightarrow $
 $$(S \bigcap \{a_p,b_p,c_p,d_p | 1 \leq p \leq k\} = \emptyset \,\vee \,S \bigcap \{a_p,b_p,c_p,d_p | k+1 \leq p \leq 2k\} = \emptyset),$$
 right hand side of implication is in direct contradiction with Lemma \ref{lem1} part $c)$.

Since if $1\leq i\leq k$ then again by Lemma \ref{lem1} part $c)$, it holds \\
S $\bigcap \{a_p,b_p,c_p,d_p | k+1 \leq p \leq 2k\}\neq 0.$ Since $0,i \notin \{p | k+1 \leq p \leq 2k\}$ then it must be $j\in \{p| k+1 \leq p \leq 2k\} $, so part II) holds.

Part III) also follow from Lemma \ref{lem1}, in similar way as part II).

We have 8 possible cases for mixed resolving set $S$.

\textbf{Case 1.} $S =\{c_0,a_i,a_j\}$.\\
Since $i\neq 0\, \wedge j\neq 0$ then
$d(a_0c_0,c_0)=d(c_0,c_0)=0$, $d(a_0c_0,a_i)=d(a_0,a_i)-1=d(c_0,a_i)$
and $d(a_0c_0,a_j)=d(a_0,a_j)-1=d(c_0,a_j)$ imply $r(a_0c_0,S)=r(c_0,S)$
which means that $S$ is not a mixed resolving set. \\

\textbf{Case 2.}$S =\{c_0,a_i,b_j\}$ \\
\textit{Subcase 1.} $1 \leq i \leq  k$ \\
From part II) it holds that $k+1\leq j\leq 2k$. Then $d(a_0c_0,c_0)=d(c_0d_{2k},c_0)=0$, $d(a_0c_0,a_i)=d(c_0,a_i)=d(c_0d_{2k},a_i)$, $d(a_0c_0,b_j)=d(c_0,b_j)=d(b_0,b_{j})+1$ and $d(c_0d_{2k},b_j)=d(d_{2k},b_j)=d(b_{2k},b_j)+2$. Since $d(b_0,b_j)=d(b_{2k},b_j)+1$ then $d(a_0c_0,b_j)=d(c_0d_{2k},b_j)$. Therefore, $r(a_0c_0,S)=r(c_0d_{2k},S)$
which means that $S$ is not a mixed resolving set.\\
\textit{Subcase 2.} $k+1 \leq i \leq  2k$ \\
From part III) it holds that $1 \leq j \leq  k$.
Then $d(a_0c_0,c_0)=d(c_0c_{1},c_0)=0$, $d(a_0c_0,a_i)=d(c_0,a_i)=d(c_0c_{1},a_i)$, $d(a_0c_0,b_j)=d(c_0,b_j)=d(b_0b_{j})+1$ and $d(c_0c_{1},b_j)=d(c_{1},b_j)=d(b_{1},b_j)+2$. Since $d(b_0,b_j)=d(b_{1},b_j)+1$ then $d(a_0c_0,b_j)=d(c_0c_{1},b_j)$. Therefore, $r(a_0c_0,S)=r(c_0c_{1},S)$
which means that $S$ is not a mixed resolving set.\\

\textbf{Case 3.} $S =\{c_0,a_i,c_j\}$ \\
\textit{Subcase 1.} $1 \leq i \leq  k$ \\
From part II) it holds that $k+1 \leq j \leq  2k$.\\ Then $d(a_0c_0,c_0)=d(c_0d_{2k},c_0)=0$, $d(a_0c_0,a_i)=d(c_0,a_i)=d(c_0d_{2k},a_i)$, $d(a_0c_0,c_j)=d(a_0,c_j)=d(c_{2k},c_j)+2$, $d(c_0d_{2k},c_j)=d(d_{2k},c_j)=d(c_{2k},c_j)+2$ and $d(a_0c_0,c_j)=d(c_0d_{2k},c_j)$.Therefore, $r(a_0c_0,S)=r(c_0d_{2k},S)$
which means that $S$ is not a mixed resolving set.\\
\textit{Subcase 2.} $k+1 \leq i \leq  2k$ \\
From part III) it holds that $1 \leq j \leq  k$.
Then, $d(a_0c_0,c_0)=d(c_0,c_0)=0$, $d(a_0c_0,a_i)=d(a_0,a_i)-1=d(c_0,a_i)$
and $d(a_0c_0,c_j)=d(a_0,a_j)-2=d(c_0,c_j)$ imply $r(a_0c_0,S)=r(c_0,S)$
which means that $S$ is not a mixed metric resolving set. \\

\textbf{Case 4.} $S =\{c_0,a_i,d_j\}$ \\
\textit{Subcase 1.} $1 \leq i \leq  k$ \\
From part II) it holds that $k+1 \leq j \leq  2k$.\\ Then $d(a_0c_0,c_0)=0=d(c_0,c_0)=0$, $d(a_0c_0,a_i)=d(a_0,a_i)-1=d(c_0,a_i)$ and
$d(a_0c_0,d_j)=d(a_0,d_j)-1=d(c_0,d_j)$ implying $r(a_0c_0,S)=r(c_0,S)$
which means that $S$ is not a mixed resolving set.\\
\textit{Subcase 2.} $k+1 \leq i \leq  2k$ \\
From part III) it holds that $1 \leq j \leq  k$.
Then $d(c_0,c_0)=d(c_0c_{1},c_0)=0$, $d(c_0,a_i)=d(c_{1},a_i)-1=d(c_0c_{1},a_i)$ and $d(c_0,d_j)=d(c_1,d_j)-1=d(c_0c_{1},d_j)$ then $r(c_0,S)=r(c_0c_{1},S)$
which means that $S$ is not a mixed resolving set.\\

\textbf{Case 5.} $S =\{c_0,b_i,a_j\}$ Reduce to Case 2. by substitution $j'=i, i'=j$.\\

\textbf{Case 6.} $S =\{c_0,b_i,b_j\}$ \\
\textit{Subcase 1.} $1 \leq i \leq  k$ \\
From part II) it holds that $k+1 \leq j \leq  2k$.\\ Then $d(a_0,c_0)=1=d(a_0d_0,c_0)$, $d(a_0,b_i)=d(b_0,b_i)+1=d(d_0,b_i)-1=d(a_0d_0,b_i)$
and $d(a_0,b_j)=d(b_0,b_j)+1=d(d_0,b_j)-1=d(a_0d_0,b_i)$ implying $r(a_0,S)=r(a_0d_0,S)$ which means that $S$ is not a mixed resolving set.\\
\textit{Subcase 2.} $k+1 \leq i \leq  2k$ Reduce to Subcase 2. by substitution $j'=i, i'=j$.\\

\textbf{Case 7.} $S =\{c_0,b_i,c_j\}$ \\
\textit{Subcase 1.} $1 \leq i \leq  k$ \\
From part II) it holds that $k+1 \leq j \leq  2k$.\\ Then $d(b_0b_1,c_0)=d(b_0,c_0)=2=d(a_1,c_0)=d(a_1b_1,c_0)$, $d(b_0b_1,b_i)=d(b_1,b_i)=d(a_1,b_i)-1=d(a_1b_1,b_i)$
and $d(b_0b_1,c_j)=d(b_0,c_j)=d(b_0,b_j)+2=d(c_0,c_j)+2=d(c_0,c_j)+d(c_0,a_1)=d(a_1,c_j)=d(b_1,c_j)-1=d(a_1b_1,c_j)$ implying $r(b_0b_1,S)=r(a_1b_1,S)$ which means that $S$ is not a mixed resolving set.\\
\textit{Subcase 2.} $k+1 \leq i \leq  2k$ \\
From part III) it holds that $1 \leq j \leq  k$.
Then $d(b_0,c_0)=2=d(d_{2k-1},c_0)$, $d(b_0,b_i)=d(b_{2k-1},b_i)+2=d(d_{2k-1},b_i)$ and
$d(b_0,c_j)=d(b_0,b_j)+2=d(c_0,c_j)+2=d(d_{2k-1},c_0)+d(c_0,c_j)=d(d_{2k-1},c_0)$ then $r(b_0,S)=r(d_{2k-1},S)$
which means that $S$ is not a mixed resolving set.\\

\textbf{Case 8.} $S =\{c_0,b_i,d_j\}$ \\
\textit{Subcase 1.} $1 \leq i \leq  k$ \\
From part II) it holds that $k+1 \leq j \leq  2k$.\\ Then $d(b_0b_1,c_0)=d(b_0,c_0)=2=d(a_1,c_0)=d(a_1b_1,c_0)$, $d(b_0b_1,b_i)=d(b_1,b_i)=d(a_1,b_i)-1=d(a_1b_1,b_i)$
and $d(b_0b_1,d_j)=d(b_0,d_j)=d(b_0,b_j)+2=d(d_0,d_j)+2=d(d_0,d_j)+d(d_0,a_1)=d(a_1,d_j)=d(b_1,d_j)-1=d(a_1b_1,d_j)$ implying $r(b_0b_1,S)=r(a_1b_1,S)$ which means that $S$ is not a mixed resolving set.\\
\textit{Subcase 2.} $k+1 \leq i \leq  2k$ \\
From part III) it holds that $1 \leq j \leq  k$.
Then $d(b_0,c_0)=2=d(d_{2k-1},c_0)$, $d(b_0,b_i)=d(b_{2k-1},b_i)+2=d(d_{2k-1},b_i)$ and
$d(b_0,d_j)=d(b_0,b_j)+2=d(d_0,d_j)+2=d(c_{2k-1},d_0)+d(d_0,d_j)=d(c_{2k-1},d_0)$ then $r(b_0,S)=r(c_{2k-1},S)$
which means that $S$ is not a mixed resolving set.\\

Since $S$ is not mixed resolving set in all eight cases, it is in contradiction with starting assumption,
so $\beta_M(J_n)\geq 4.$
Therefore, from the previous three steps, the proof of theorem is completed.
\end{proof}

It would be interesting to make comparison between mixed metric dimension and metric dimension for flower snarks.
For $n=5$ situation is easy since all three dimensions can be obtained by a total enumeration,
so $\beta(J_{n})=3<\beta_{E}(J_{n})=4<\beta_{M}(J_{n})=5$.
For odd $n\geq 7$, mixed metric dimension is larger than metric dimension, i.e. $\beta(J_{n})=3<\beta_{M}(J_{n})=4$.
From the Theorem \ref{fs} it is easy to see, that for $n \geq 7$, similar to metric dimension, mixed metric dimension for
flower snarks is constant,  i.e. it doesn't depend on $n$.

When we consider edge metric dimension, situation is not so clear.
Theorem \ref{fs} has obvious corollary, that for $n \geq 7$ it is $\beta_E(J_n) \leq 4$.
From the Theorem \ref{edim1} from \cite{bib8}, lower bound of edge metric dimension applied to
flower snarks is $\beta_E(J_n) \geq 3$, so it holds $(\forall n \geq 7) \beta_E(J_n) \in \{3, 4\}$.
It would be interesting to find exact value.

\subsection{Mixed metric dimension of wheels graphs}

In the following, it will be obtained the mixed metric dimension  of wheels graphs.  In the next theorem, the mixed metric dimension of these graphs is determined.
\begin{thm}$$\beta_{M}(W_{n})=\left\{
                     \begin{array}{ll}
                       4, & \hbox{n=3;} \\
                       n, & \hbox{$n\geq 4$.}
                     \end{array}
                   \right.$$

\end{thm}
Proof: For $n=3$ using total enumeration, we have $\beta_{M}(W_{3})=4=|W_3|$ so that $S=V(W_3).$\\
Therefore, mixed metric dimension of wheels will be considered only for case when $n\geq 4$.  In order to present it more clearly, the representation of mixed metric coordinates will be given in Table \ref{tab3} of each
vertex and each edge with respect to $V(W_n)$.

\begin{table}
\begin{center}
\caption{Mixed metric representations for $W_n$}
\label{tab3}
\vspace{0.3cm}
\begin{tabular}{|c|c|c|}
  \hline
  vetex & cond. & $r(v,V)$\\
  \hline

  $v_{0}$ &   & $(0,1,...,1)$ \\
 $v_{1}$ &  & ($1,0,1,2...,2,1$) \\
 $v_{2}$ &  & ($1,1,0,1,2...,2$) \\
  $v_{i}$ & $3 \leq i \leq n-2$  & ($1,2,...,2,1,\underbrace{0}_{i},1,2,...,2$) \\
  $v_{n-1}$ &  & ($1,2,...,2,1,0,1$) \\
  $v_{n}$ &  & ($1,1,2,...,2,1,0$) \\

  \hline
edge & cond. & $r(e,V)$\\
\hline
 $ v_0v_{i}$ &  $1 \leq i \leq n$ & ($0,1,...,1,\underbrace{0}_{i},1,...,1$) \\

      \hline
   $ v_1v_{2}$ &   & ($1,0,0,1,2,...,2,1$) \\
  $ v_1v_{n}$ &   & ($1,0,1,2,...,2,1,0$) \\
  $ v_2v_{3}$ &   & ($1,1,0,0,1,2,...,2$) \\
  $v_{i}v_{i+1}$& $3 \leq i \leq n-3$ & ($1,2,...,2,1,\underbrace{0}_{i},0,1,2,...,2$) \\
  $v_{n-2}v_{n-1}$& & ($1,2,...,2,1,0,0,1$) \\
 $v_{n-1}v_{n}$& & ($1,1,2,...,2,1,0,0$) \\
\hline
\end{tabular}
\end{center}
\end{table}

\vspace{0.0cm}
Two steps will be considered.\\

\textbf{\underline{Step 1}:}  {\em Upper bound for $n \geq 4$}.  \\
Let $S=\{v_i| 1\leq i\leq n\}$ and it will be proved that $S$ is mixed resolving set of $W_{n}$ for $n\geq 4.$ Since $S=V \setminus \{v_0\}$, then for each item, vector of metric coordinates with respect to $S$ is presented in the Table \ref{tab4}.

\begin{table}
\begin{center}
\caption{Mixed metric representations for $W_n$}
\label{tab4}
\vspace{0.3cm}
\begin{tabular}{|c|c|c|}
  \hline
  vetex & cond. & $r(v,V\setminus\{v_0\})$\\
  \hline

  $v_{0}$ &   & $(1,...,1)$ \\
 $v_{1}$ &  & ($0,1,2...,2,1$) \\
 $v_{2}$ &  & ($1,0,1,2...,2$) \\
  $v_{i}$ & $3 \leq i \leq n-2$  & ($2,...,2,1,\underbrace{0}_{i},1,2,...,2$) \\
  $v_{n-1}$ &  & ($2,...,2,1,0,1$) \\
  $v_{n}$ &  & ($1,2,...,2,1,0$) \\

  \hline
edge & cond. & $r(e,V\setminus\{v_0\})$\\
\hline
 $ v_0v_{i}$ &  $1 \leq i \leq n$ & ($1,...,1,\underbrace{0}_{i},1,...,1$) \\

      \hline
   $ v_1v_{2}$ &   & ($0,0,1,2,...,2,1$) \\
  $ v_1v_{n}$ &   & ($0,1,2,...,2,1,0$) \\
  $ v_2v_{3}$ &   & ($1,0,0,1,2,...,2$) \\
  $v_{i}v_{i+1}$& $3 \leq i \leq n-3$ & ($2,...,2,1,\underbrace{0}_{i},0,1,2,...,2$) \\
  $v_{n-2}v_{n-1}$& & ($2,...,2,1,0,0,1$) \\
 $v_{n-1}v_{n}$& & ($1,2,...,2,1,0,0$) \\
\hline
\end{tabular}
\end{center}
\end{table}

When $n\geq 4 $ there is at least one coordinate equal to 2. According to the previous, it follows that mixed metric coordinates of all items  are mutually different. Therefore, $S$ is a mixed resolving set, so it holds $\beta_{M}(W_{n})\leq n$.\\

\textbf{\underline{Step 2}:}  {\em Lower bound for $n \geq 4$}.  \\
Assume the opposite that it is $\beta_{M}(W_{n})\leq n-1$. Then mixed resolving set $S$ exists, so that $|S|\leq n-1$.

\textbf{Case 1.} \underline{$v_0 \notin  S$}.

Since $|V|=n+1 \,\, \wedge \,\, |S|\leq n-1 \,\, \Rightarrow (\exists i)\,\, 1\leq i \leq n, \, v_i\notin S.$  From the previous holds that $r(v_0,S)=r(v_0v_i, S)=(1,...,1)$. It wil be concluded that $S$ is not mixed resolving set, which is a contradiction with a starting assumption.

\textbf{Case 2.} \underline{$v_0 \in S$}.

Since $|V|=n+1 \,\, \wedge \,\, |S|\leq n-1 \,\, \Rightarrow (\exists i,j)\,\, 1\leq i,j \leq n, \, v_i,v_j\notin S.$  From the previous holds that $r(v_0,S)=r(v_0v_i, S)=(0,1,...,1).$ It will be concluded that $S$ is not mixed resolving set, which is a contradiction  with a starting assumption.

Since  $S$ is not mixed resolving set in both cases, it follows that $\beta_{M}(W_{n})\geq n.$

Therefore, from Step 1 and Step 2, it follows that $\beta_{M}(W_{n})= n.$

\begin{flushright}
$\Box$
\end{flushright}

It should be noted that Step 1  could  also be indirectly proved using Proposition \ref{exc} from \cite{bib8}. In order to give constructive proof, we have decided on the proof presented above.
\vspace{0,2cm}

As with the prism graphs, all three previously mentioned metric invariants could be compared. It is already known from \cite{bib10} that $$\beta(W_n)=\left\{
                     \begin{array}{ll}
                       3, & \hbox{n=3,6} \\
                       2, & \hbox{n=4,5} \\
                       \lfloor\frac{2n+2}{5}\rfloor, & \hbox{$n\geq 6$}

                     \end{array}
                   \right.$$
 and from   \cite{bib11}, it follows that $$\beta_{E}(W_n)=\left\{
                     \begin{array}{ll}
                       n, & \hbox{n=3,6} \\
                       n-1, & \hbox{$n\geq 5$}

                     \end{array}
                   \right..$$

From the previous, it follows that edge metric dimension of wheel graphs is strictly larger than the value for the metric dimension, except for $n=3$ when these dimensions are equal. By comparing to the mixed metric dimension obtained in Theorem 2.2., it follows that mixed metric dimension is strictly larger than the value for the edge metric dimension, except for $n=4$, when these dimensions are equal. Unlike the prism graphs described above, for wheels graphs, mixed metric dimension depends on $ n $.

\section{Conclusions}

In this paper, mixed metric dimension for flower snarks and wheel graphs is considered.
First, it is given lemma about some properties of mixed resolving sets of
flower snark graphs. Next, it is used for obtaining exact value,
which flower snarks it is constant and equal to 4, for $n \geq 7$.
Last, it is present mixed metric dimension for wheels, and it is proved to be equal to $n$, for $n \leq 4$,
while $\beta_{M}(W_{3})=4$.

Further work can be directed in finding mixed metric dimension of some other interesting classes of graphs.
Other direction could be finding exact value of edge metric dimension of flower snarks.

\section*{Acknowledgements}

The author is grateful to Jozef Kratica for its constructive and useful comments which have helped to improve the quality of this paper greatly.

\end{document}